\documentclass{article}
\usepackage[utf8]{inputenc}
\usepackage{amsthm, amssymb, mathtools, hyperref, cases, enumerate, xcolor, accents, authblk, mathrsfs}
\usepackage[inline]{enumitem}

\usepackage[backend=biber]{biblatex}
\AtEveryBibitem{\clearlist{language}}
\numberwithin{equation}{section}
\newtheorem{thm}{Theorem}[section]
\newtheorem{prop}[thm]{Proposition}
\newtheorem{lem}[thm]{Lemma}

\newtheorem{rmk}[thm]{Remark}
\theoremstyle{definition}

\newcommand{\nat}{\mathbb{N}}
\newcommand{\real}{\mathbb{R}}

\newcommand{\sphere}{\mathscr{S}}

\newcommand{\complex}{\mathbb{C}}
\newcommand{\iu}{\mathrm{i}}

\newcommand{\op}{- \Delta_\alpha}
\newcommand{\uop}{- \accentset{\circ}{\Delta}_\alpha}

\newcommand{\VS}{\mathrm{VS}}
\newcommand{\GN}{\mathrm{GN}}
\newcommand{\HLS}{\mathrm{HLS}}

\newcommand{\C}{\mathcal{C}}
\newcommand{\E}{\mathcal{E}}

\newcommand{\calH}{\mathcal{H}}

\newcommand{\Q}{\mathcal{Q}}
\newcommand{\V}{\mathcal{V}}

\newcommand{\eps}{\varepsilon}

\newcommand{\loc}{\mathrm{loc}}
\newcommand{\dif}{\mathrm{d}}

\DeclareMathOperator{\Dom}{Dom}

\DeclarePairedDelimiter{\abs}{\lvert}{\rvert}
\DeclarePairedDelimiter{\norm}{\lVert}{\rVert}
\DeclarePairedDelimiter{\parens}{(}{)}
\DeclarePairedDelimiter{\set}{\{}{\}}
\DeclarePairedDelimiter{\brackets}{\lbrack}{\rbrack}

\DeclarePairedDelimiter{\angles}{\langle}{\rangle}
\DeclarePairedDelimiter{\coi}{\lbrack}{\lbrack}
\DeclarePairedDelimiter{\oci}{\rbrack}{\rbrack}
\DeclarePairedDelimiter{\ooi}{\rbrack}{\lbrack}

\title{Ground states of the planar nonlinear Schrödinger--Newton system with a point interaction}
\begin{document}

\author{Gustavo de Paula Ramos\thanks{gpramos@icmc.usp.br}}
\affil{Instituto de Ciências Matemáticas e de Computação, Universidade de São Paulo, Avenida Trabalhador São-Carlense, 400, 13566-590 São Carlos SP, Brazil}

\date{\today}
\maketitle
\begin{abstract}
\sloppy
We establish sufficient conditions for the existence of ground states of the following  normalized nonlinear Schrödinger--Newton system with a point interaction:
\[
\begin{cases}
- \Delta_\alpha u
=
w u
+
\beta u |u|^{p - 2}
&\text{on} ~ \mathbb{R}^2;
\\
- \Delta w
=
2 \pi |u|^2
&\text{on} ~ \mathbb{R}^2;
\\
\|u\|_{L^2}^2 = c,
\end{cases}
\]
where $p > 2$; $\alpha, \beta \in \real$ and
$- \Delta_\alpha$ denotes the Laplacian of point interaction with scattering length
$(- 2 \pi \alpha)^{- 1}$.
Additionally, we show that critical points of the corresponding constrained energy functional are naturally associated with standing waves of the evolution problem
\[
\mathrm{i} \psi' (t)
=
- \Delta_\alpha \psi (t)
-
(\log |\cdot| \ast |\psi (t)|^2)
\psi (t)
-
\beta
\psi \parens{t}
|\psi (t)|^{p - 2}.
\]

\smallskip
\sloppy \noindent \textbf{Keywords.} nonlinear Schrödinger equation, zero-range potential, delta interaction, normalized solutions
\end{abstract}
\tableofcontents

\section{Introduction}

\subsection{Considered problem, context and motivation}
\label{sect:problem}
This paper is concerned with the following constrained nonlinear Schrödinger--Newton system with a point interaction:
\begin{equation}
\label{eqn:delta-SN}
\begin{cases}
\op u
=
w u
+
\beta u \abs{u}^{p - 2}
&\text{on} ~ \mathbb{R}^2;
\\
- \Delta w = 2 \pi \abs{u}^2
&\text{on} ~ \mathbb{R}^2;
\\
\norm{u}_{L^2}^2 = c,
\end{cases}
\end{equation}
where $p > 2$; $\alpha, \beta \in \real$;
$\op$ denotes the Laplacian of point interaction with scattering length
$\parens{- 2 \pi \alpha}^{- 1}$
and we want to solve for
$u \colon \real^2 \to \complex$,
$w \colon \real^2 \to \coi{0, \infty}$.

\sloppy
Let us briefly explain what we mean by a \emph{point interaction}. By
$
\set{\uop}_{
	\alpha \in \oci{- \infty, \infty}
}
$,
we denote the family of $L^2$-self-adjoint extensions of the closure of the operator
$
- \Delta|_{
	C_0^\infty \parens{\real^2 \setminus \set{0}}
}
$
in $L^2$ with the goal of obtaining $L^2$-self-adjoint operators that mimic the behavior of
$- \Delta + \frac{1}{\alpha} \delta_0$
in $C_c^\infty$. We deviate from the usual notation for technical reasons, letting $\op$ denote a natural extension of $\uop$ to a certain linear subspace of $L^2_{\loc}$. For more details, see Section \ref{sect:Laplacian}.

Consider the following semilinear system:
\begin{equation}
\label{eqn:SN}
\begin{cases}
- \Delta u + \omega u
=
w u
+
\beta u \abs{u}^{p - 2}
&\text{in} ~ \real^N;
\\
- \Delta w = C_N \abs{u}^2
&\text{in} ~ \real^N,
\end{cases}
\end{equation}
where
\[
C_N
:=
\begin{cases}
2 \pi,
&\text{if} ~ N = 2;
\\
N \parens{N - 2} L_N,
&\text{if} ~ N \geq 3;
\end{cases}
\]
$L_N$ denotes the Lebesgue measure of the unit ball in $\real^N$ and we want to solve for
$u \colon \real^N \to \complex$,
$w \colon \real^N \to \coi{0, \infty}$
and $\omega \in \real$. The usual approach to study \eqref{eqn:SN} involves considering the fundamental solution of $-\Delta$, that is, the function
$
\Phi_N
\colon
\real^N \setminus \set{0} \to \real
$
defined as
\[
\Phi_N \parens{x}
=
\begin{cases}
-\frac{1}{2 \pi}
\log \abs{x},
&\text{if} ~ N = 2;
\\
\frac{1}{C_N \abs{x}^{N - 2}},
&\text{if} ~ N \geq 3.
\end{cases}
\]
The identity
$- \Delta \Phi_N = \delta_0$
holds in the sense of distributions, so the function
\[
w_u
:=
C_N
\Phi_N \ast \abs{u}^2
=
C_N
\int
	\Phi_N \parens{\cdot - y}
	\abs*{u \parens{y}}^2
\dif y
\]
formally satisfies
$- \Delta w_u = C_N \abs{u}^2$.
As such, we turn our attention to the following integrodifferential equation:
\begin{equation}
\label{eqn:reduced-SN}
- \Delta u + \omega u
=
w_u u + \beta u \abs{u}^{p - 2}
\quad \text{in} \quad \real^N.
\end{equation}
At least formally, \eqref{eqn:reduced-SN} is associated with the energy functional
\begin{equation}
\label{eqn:calE}
\E \parens{u}
:=
\int_{\real^N}
	\frac{1}{2}
	\abs*{\nabla u \parens{x}}^2
	-
	\frac{1}{4}
	w_u \parens{x}
	\abs*{u \parens{x}}^2
	-
	\frac{\beta}{p}
	\abs*{u \parens{x}}^p
\dif x.
\end{equation}

When $N = 3$ and $\beta = 0$, \eqref{eqn:reduced-SN} models the spatial profile of standing waves of quantum matter with energy $\omega$ under the sole effect of a nonlocal self-attractive interaction, appearing in numerous physical contexts (see \cite{pekarUntersuchungenUberElektronentheorie1954,liebExistenceUniquenessMinimizing1977, penroseGravitysRoleQuantum1996}).

On one hand, \eqref{eqn:reduced-SN} has already been widely studied when $N \geq 3$ and
$2 \leq p \leq 2 N / \parens{N - 2}$, in which case it follows from the Hardy--Littlewood--Sobolev inequality and the Sobolev embeddings that $\E$ is well defined in the usual Sobolev space $W^{1, 2} \parens{\real^N}$. We refer the reader to Moroz \& Van Schaftingen's
\cite{morozGroundstatesNonlinearChoquard2013, morozGuideChoquardEquation2017}
for an overview of the current knowledge about this problem when $N \geq 3$.

On the other hand, \eqref{eqn:reduced-SN} is more mathematically challenging and the corresponding literature is noticeably scarcer in the case
$N = 2$. First of all, $\Phi_2$ is not bounded from above or below, so the nonlinear operator
$u \mapsto - \Delta u + w_u u$
is not naturally associated with a positive definite quadratic form. Furthermore, \eqref{eqn:calE} is not well defined on the natural Sobolev space
$W^{1, 2} \parens{\real^2}$. Due to these peculiarities, there is an increasing interest and a rapidly expanding literature concerned with this problem. To cite a few papers, we highlight Stubbe's \cite{stubbeBoundStatesTwodimensional2008}, which introduced a convenient variational framework to treat this problem (see Section \ref{sect:vrtnl-frmwrk}); Masaki's \cite{masakiEnergySolutionSchrodinger2011}, which established the well-posedness of the associated Cauchy problem in the energy space;
\cite{cingolaniPlanarSchrodingerPoisson2016, duGroundStatesHigh2017} obtained several existence and multiplicity results; Cingolani \& Jeanjean's \cite{cingolaniStationaryWavesPrescribed2019} established existence and non-existence results for the normalized problem
\[
\begin{cases}
- \Delta u = w_u u + \beta u \abs{u}^{p - 2}
&\text{on} ~ \real^2;
\\
\norm{u}_{L^2}^2 = c,
\end{cases}
\]
and \cite{alvesExistenceNormalizedSolutions2023, chenMultipleNormalizedSolutions2024} obtained existence and multiplicity results for the normalized problem with exponential critical growth.

Even though the considered model for point interactions was introduced by Berezin \& Faddeev in the 1960s (see \cite{berezinRemarkSchrodingerEquation1961}) and a comprehensive book on the subject was already available by the end of the 1980s (see Albeverio, Gesztesy, Høegh-Krohn \& Holden's \cite{albeverioSolvableModelsQuantum1988}), the study of semilinear nonlocal problems involving a point interaction in dimensions 2 and 3 is much more recent. We proceed to a quick review of the relevant literature. Consider the \emph{Hartree equation} with a point interaction,
\[
\iu \psi' \parens{t}
=
\uop \psi \parens{t}
-
\parens*{w \ast \abs*{\psi \parens{t}}^2}
\psi \parens{t}
\quad \text{in} \quad
\real^N.
\]
In \cite{michelangeliSingularHartreeEquation2021}, Michelangeli, Olgiati \& Scandone obtained results about the well-posedness of the associated Cauchy problem in dimension $N = 3$ when $\alpha \geq 0$ under hypotheses on the integrability of $w$. In \cite{georgievStandingWavesGlobal2024}, Georgiev, Michelangeli \& Scandone obtained action ground states, their qualitative properties and established the well-posedness for the associated Cauchy problem in dimension $N  = 2$ when $w$ is non-negative, radial, non-increasing and sufficiently integrable. The analogous system to \eqref{eqn:delta-SN} obtained when considering self-repulsion instead of self-attraction is the \emph{nonlinear Schrödinger--Maxwell system} (often called the \emph{nonlinear Schrödinger--Poisson system}) with a point interaction:
\[
\begin{cases}
\uop u + \omega u + w u
=
\beta u \abs{u}^{p - 2}
&\text{in} ~ \Omega;
\\
- \Delta w = \omega_N \abs{u}^2
&\text{in} ~ \Omega.
\end{cases}
\]
(see \cite{daveniaNonradiallySymmetricSolutions2002}
for a physical deduction of this system when
$\alpha = \infty$). The Dirichlet problem where
$\beta = 0$ and $\Omega$ denotes a smooth bounded subset of $\real^3$ was considered by Coclite \& Holden in \cite{cocliteGroundStatesSchrodingerMaxwell2010, cocliteSchrodingerMaxwellSystem2007}, where the existence of nontrivial weak solutions and ground states were established. Furthermore, the normalized problem in $\Omega = \real^3$ was recently considered in the author's preprint \cite{depaularamosMinimizersMassconstrainedFunctionals2024}, where the existence of ground states with sufficiently small mass was proposed.

In this context, this paper is motivated by the fact that, to the best of our knowledge, there are no studies about the planar nonlinear Schrödinger--Newton system with a point interaction. More precisely, our main goal is to prove the existence of ground states of
\[
\begin{cases}
\op u = w_u u + \beta u \abs{u}^{p - 2}
&\text{on} ~ \real^2;
\\
\norm{u}_{L^2}^2 = c,
\end{cases}
\]
(see Theorem \ref{thm:ground-states}).

We finish with a comment on the organization of the rest of the introduction.
\begin{itemize}
\item
In Section \ref{intro:notation}, we fix the notation used throughout the paper.
\item
In Section \ref{sect:Laplacian}, we recall the precise definition of $\op$.
\item
In Section \ref{sect:vrtnl-frmwrk}, we introduce the relevant variational framework.
\item
In Section \ref{intro:main-results}, we state our main results.
\end{itemize}

\subsection{Notation}
\label{intro:notation}
\begin{itemize}
\item
The open ball on $\real^2$ centered at $x$ with radius $r > 0$ is denoted by $B_r \parens{x}$.
\item
Unless denoted otherwise, we consider complex functional spaces of functions defined a.e. on $\real^2$. Nonlinear functionals on these spaces are denoted by uppercase calligraphic Latin letters
$\mathcal{A}, \mathcal{B}, \mathcal{C} \ldots$
\item
Sesquilinear forms are linear in the second entry and we enclose their argument with square brackets. In the case of inner products, we most often employ the conventional notation
$\angles{\cdot, \cdot}$.
\item
The Sobolev space $W^{k, p}$ (resp., homogeneous Sobolev space $\dot{W}^{k, p}$) contains a.e.-defined functions whose weak derivatives of order
$0, \ldots, k$ (resp., of order $k$) are in $L^p$.
\item
If $X$ is a Hilbert space, then we denote its topological-linear dual by $X^*$ and the following notation is employed for the real part of the duality pairing:
\[
\angles{T, x}_{X^*, X}
:=
\Re \parens{T x}
\]
for every $T \in X^*$ and $x \in X$.
\item
Unless mention otherwise, we suppose that:
\begin{itemize}
\item
$\alpha, \beta$ denote real numbers;
\item
$c, \lambda$ denote positive numbers.
\end{itemize}
\end{itemize}

\subsection{The Laplacian of point interaction}
\label{sect:Laplacian}
\subsubsection{Definition}
We need to introduce a family of Green's functions to rigorously define the Laplacian of point interaction. Let
$
G_\lambda
\colon
\real^2 \setminus \set{0}
\to
\ooi{0, \infty}
$
be given by
\[
G_\lambda \parens{x}
=
\frac{1}{2 \pi}
K_0 \parens*{\sqrt{\lambda} \abs{x}},`
\]
so that
$
- \Delta G_\lambda + \lambda G_\lambda
=
\delta_0
$
in the sense of distributions, where $K_0$ denotes a modified Bessel function of the second kind (see \cite[Section 9.6]{abramowitzHandbookMathematicalFunctions1972}). The Green's function $G_\lambda$ has exponential decay at infinity and
\[
\frac{G_\lambda \parens{x}}
	{
		- \frac{\log \parens{\sqrt{\lambda} \abs{x}}}{2 \pi}
	}
\xrightarrow[x \to 0]{}
1,
\]
so
\begin{equation}
\label{eqn:integrability-of-G_lambda}
G_\lambda \in L^r
\quad \text{for every} \quad
r \in \coi{2, \infty}.
\end{equation}
Following \cite[Chapter I.5]{albeverioSolvableModelsQuantum1988}, the self-adjoint operator
$\uop \colon \Dom \parens{\uop} \to L^2$
is defined as
\[
\uop u = - \Delta \phi - q \lambda G_\lambda
\]
for every
$
u
=
\phi + q G_\lambda \in \Dom \parens{\uop}
$,
where
\[
\Dom \parens{\uop}
:=
\set*{
	\phi + q G_\lambda:
	\phi \in W^{2, 2}, ~
	q \in \complex, ~
	\lambda > 0
	\quad \text{and} \quad
	\parens{\alpha + \theta_\lambda} q
	=
	\phi \parens{0}
},
\]
\[
\theta_\lambda
:=
\frac{\gamma}{2 \pi}
+
\frac{1}{2 \pi}
\log \parens*{\frac{\sqrt{\lambda}}{2}}
\]
and
\[
\gamma
:=
\lim_{n \to \infty} \parens*{
	-
	\log n
	+
	\sum_{1 \leq k \leq n} \frac{1}{k}
}
\in \ooi{0.57, 0.58}
\]
is the \emph{Euler--Mascheroni constant}.

\subsubsection{Nonuniqueness of decomposition}
\label{section:nonuniqueness}

An important aspect of $\Dom \parens{\uop}$ is the nonuniqueness of decomposition of its elements. More precisely, consider the set
\[
D
:=
\set*{
	\phi + q G_\lambda :
	\phi \in W^{1, 2}_\loc, \quad
	q \in \complex
	\quad \text{and} \quad
	\lambda > 0
},
\]
which clearly contains $\Dom \parens{\uop}$. By \emph{nonuniqueness of decomposition}, we mean the result stated below.

\begin{prop}
\label{prop:nonuniqueness}
\begin{enumerate}
\item
Given $u \in D$, there exists a unique complex number $\Q \parens{u}$ (called the \emph{charge} of $u$) for which if $\phi \in W^{1, 2}_\loc$,
$q \in \complex$ and $\lambda > 0$ are such that
$u = \phi + q G_\lambda$, then
$q = \Q \parens{u}$.
\item
Given $u \in D$ and $\lambda > 0$, there exists a unique $\phi_\lambda \in W^{1, 2}_\loc$ such that
$u = \phi_\lambda + \Q \parens{u} G_\lambda$.
\end{enumerate}
\end{prop}

In dimension 3, this result is a corollary of \cite[Lemma 2.1]{depaularamosMinimizersMassconstrainedFunctionals2024}. The proof in dimension 2 follows from the same arguments.

\subsubsection{The associated energy space}
\label{sect:energy-space}
The spectrum of $\uop$ is given by
$
\sigma \parens{\uop}
=
\set{- \omega_\alpha}
\cup
\coi{0, \infty}
$,
where
\[
\omega_\alpha
:=
4 e^{- 4 \pi \alpha - 2 \gamma}
>
0
\]
(see \cite[Theorem 5.4]{albeverioSolvableModelsQuantum1988}). In particular, we obtain the following inner product on $\Dom \parens{\uop}$:
\begin{equation}
\label{eqn:pre-inner-product}
\parens{u, v}
\mapsto
\angles{\uop u, v}_{L^2}
+
\parens{1 + \omega_\alpha}
\angles{u, v}_{L^2}.
\end{equation}
We define $W^{1, 2}_\alpha$ as the Hilbert space obtained as completion of $\Dom \parens{\uop}$ with respect to \eqref{eqn:pre-inner-product} and we call $W^{1, 2}_\alpha$ the \emph{energy space} associated with $\uop$.

The underlying vector space of $W^{1, 2}_\alpha$ is often called the \emph{form domain} of $\uop$ and is explicitly given by
\[
\Dom \brackets{\uop}
:=
\set{
	\phi + q G_\lambda :
	\phi \in W^{1, 2}, \quad
	q \in \complex
	\quad \text{and} \quad
	\lambda > 0
}.
\]
The inner product of $W^{1, 2}_\alpha$ is given by
\[
\angles{u, v}_{W^{1, 2}_\alpha}
:=
H_\alpha \brackets{u, v}
+
\parens{1 + \omega_\alpha}
\angles{u, v}_{L^2},
\]
where
\[
H_\alpha \brackets{u, v}
:=
\angles{\phi_u, \phi_v}_{\dot{W}^{1, 2}}
+
\lambda \parens*{
	\angles{\phi_u, \phi_v}_{L^2}
	-
	\angles{u, v}_{L^2}
}
+
\parens{\alpha + \theta_\lambda}
\overline{q_u} q_v
\]
for every $u = \phi_u + q_u G_\lambda$ and
$v = \phi_v + q_v G_\lambda$ in
$\Dom \brackets{\uop}$. More precisely, the Hermitian form $H_\alpha$ denotes the unique continuous extension to $\Dom \brackets{\uop}$ of
\begin{equation}
\label{eqn:-Delta_alpha-u,u}
\Dom \parens{\uop} \times \Dom \parens{\uop}
\ni
\parens{u, v}
\mapsto
\angles{\uop u, v}_{L^2}.
\end{equation}

\begin{rmk}
\begin{enumerate}
\item
For a proof that $H_\alpha$ is well defined, see \cite[Lemma 2.4]{depaularamosMinimizersMassconstrainedFunctionals2024}).
\item
The identity
\[
\angles{
	- \Delta \phi + \lambda \phi,
	G_\lambda
}_{L^2}
=
\phi \parens{0}
\]
for every $\lambda > 0$ and $\phi \in W^{2, 2}$
(see Lemma \ref{lem:Dirac-identity}) is essential to prove that $H_\alpha$ extends the Hermitian form \eqref{eqn:-Delta_alpha-u,u}. For an analogous identity in $\real^3$, see \cite[Lemma 2.3]{depaularamosMinimizersMassconstrainedFunctionals2024}.
\end{enumerate}
\end{rmk}

For notational reasons, it will be useful to define the nonlinear functional
$\calH_\alpha \colon W^{1, 2}_\alpha \to \real$
as
\[
\calH_\alpha \parens{u}
=
H_\alpha \brackets{u, u}
=
\norm{\phi}_{\dot{W}^{1, 2}}^2
+
\lambda \parens*{
	\norm{\phi}_{L^2}^2
	-
	\norm{u}_{L^2}^2
}
+
\parens{\alpha + \theta_\lambda}
\abs{q}^2
\]
for every
$
u = \phi + q G_\lambda \in \Dom \brackets{\uop}
$.
To finish, we highlight that
$W^{1, 2}_\alpha \hookrightarrow L^r$
for every
$r \in \coi{2, \infty}$.
Indeed, it is well known that the Sobolev space
$W^{1, 2}$ is canonically embedded in $L^r$ for every $r \in \coi{2, \infty}$. Due to \eqref{eqn:integrability-of-G_lambda}, the analogous embeddings hold for $W^{1, 2}_\alpha$.

\subsubsection{An extension of the Laplacian of point interaction}
\label{sect:extension}
Consider the natural extension of
$\uop$ obtained as follows: let the linear operator
$\op \colon \Dom \parens{\op} \to L^2_{\loc}$
be defined as taking
$u = \phi + q G_\lambda \in \Dom \parens{\op}$
to
\[\op u = - \Delta \phi - q \lambda G_\lambda,\]
where
\[
\Dom \parens{\op}
:=
\set*{
	\phi + q G_\lambda:
	\phi \in W^{2, 2}_{\loc}, ~
	\lambda > 0
	\quad \text{and} \quad
	\parens{\alpha + \theta_\lambda} q
	=
	\phi \parens{0}
}.
\]
In Section \ref{intro:main-results}, we comment on the motivation for considering this extension.

\subsection{Variational framework}
\label{sect:vrtnl-frmwrk}

The strategy adopted in Stubbe's \cite{stubbeBoundStatesTwodimensional2008} to obtain a functional space where \eqref{eqn:calE} is well defined consisted in introducing the Hilbert space
$
X := \parens{
	\VS \parens{X}, \angles{\cdot, \cdot}_X
}
$,
where
\[
\VS \parens{X}
:=
\set*{
	\phi \in W^{1, 2}:
	\int
		\log \parens*{1 + \abs{x}}
		\abs*{\phi \parens{x}}^2
	\dif x
	<
	\infty
}
\]
and
\[
\angles{\phi, \psi}_X
:=
\angles{\phi, \psi}_{W^{1, 2}}
+
\int
	\log \parens*{1 + \abs{x}}
	\overline{\phi \parens{x}}
	\psi \parens{x}
\dif x.
\]
Let us briefly explain why $\E$ is well defined on $X$. Elementary computations show that
\[
\log t
=
\log \parens{1 + t}
-
\log \parens*{1 + \frac{1}{t}}
\]
for every $t > 0$. In particular,
\begin{multline*}
\int \int
	\log \parens*{\abs{x - y}}
	\abs*{\phi \parens{x}}^2
	\abs*{\phi \parens{y}}^2
\dif x \dif y
=
\\
\begin{aligned}
&=
\int \int
	\log \parens*{1 + \abs{x - y}}
	\abs*{\phi \parens{x}}^2
	\abs*{\phi \parens{y}}^2
\dif x \dif y
\\
&-
\int \int
	\log \parens*{1 + \frac{1}{\abs{x - y}}}
	\abs*{\phi \parens{x}}^2
	\abs*{\phi \parens{y}}^2
\dif x \dif y.
\end{aligned}
\end{multline*}
On one hand, the first integral on the RHS is well defined due to the definition of $X$. On the other hand, the second integral on the RHS is also well defined due to the elementary inequality
\[
\log \parens*{1 + \frac{1}{t}}
\leq
\frac{1}{t}
\]
for every $t > 0$ and the Hardy--Littlewood--Sobolev inequality (see Proposition \ref{prop:HLS}).

Inspired by this construction, we introduce the Hilbert space
$
X_\alpha
:=
\parens{
	\VS \parens{X_\alpha},
	\angles{\cdot, \cdot}_{X_\alpha}
}
$,
where
\[
\VS \parens{X_\alpha}
:=
\set*{
	u \in W^{1, 2}_\alpha:
	\int
		\log \parens*{1 + \abs{x}}
		\abs*{u \parens{x}}^2
	\dif x
	<
	\infty
}
\]
and
\[
\angles{u, v}_{X_\alpha}
:=
\angles{u, v}_{W^{1, 2}_\alpha}
+
\int
	\log \parens*{1 + \abs{x}}
	\overline{u \parens{x}}
	v \parens{x}
\dif x.
\]

\begin{rmk}
Equivalently,
\[
\VS \parens{X_\alpha}
=
\set{
	\phi + q G_\lambda:
	\phi \in X, \quad
	q \in \complex
	\quad \text{and} \quad
	\lambda > 0
}
\]
because the Green's function $G_\lambda$ is square-integrable and has exponential decay at infinity, so
$
\int
	\log \parens*{1 + \abs{x}}
	G_\lambda \parens{x}^2
\dif x
<
\infty
$.
\end{rmk}

As such, we naturally associate the integrodifferential equation
\[\op u = w_u u + \beta u \abs{u}^{p - 2}\]
with the \emph{energy functional}
$\E_\alpha \colon X_\alpha \to \real$
defined as
\[
\E_\alpha \parens{u}
=
\frac{1}{2}
\mathcal{H}_\alpha \parens{u}
+
\frac{1}{4}
\V_0 \parens{u}
-
\frac{\beta}{p}
\C \parens{u},
\]
where
\[
\V_0 \parens{u}
:=
\int \int
	\log \parens*{\abs{x - y}}
	\abs*{u \parens{x}}^2
	\abs*{u \parens{y}}^2
\dif x \dif y
\]
and
$\C \parens{u} := \norm{u}_{L^p}^p$.
Inspired by
\cite{cingolaniStationaryWavesPrescribed2019, cingolaniPlanarSchrodingerPoisson2016},
we also define the nonlinear functionals
$\V_1 \colon X_\alpha \to \coi{0, \infty}$,
$\V_2 \colon L^{\frac{8}{3}} \to \coi{0, \infty}$
as given by
\[
\V_1 \parens{u}
=
\int \int
	\log \parens*{1 + \abs{x - y}}
	\abs*{u \parens{x}}^2
	\abs*{u \parens{y}}^2
\dif x \dif y
\]
and
\[
\V_2 \parens{u}
=
\int \int
	\log \parens*{1 + \frac{1}{\abs{x - y}}}
	\abs*{u \parens{x}}^2
	\abs*{u \parens{y}}^2
\dif x \dif y,
\]
so that $\V_0 = \V_1 - \V_2$ in $X_\alpha$. To finish, we highlight a few results that follow from the arguments in the proof of
\cite[Lemma 2.2]{cingolaniPlanarSchrodingerPoisson2016}.

\begin{lem}
\label{lem:CW-2.2}
\begin{enumerate}
\item
The Hilbert space $X_\alpha$ is compactly embedded in $L^r$ for every
$r \in \coi{2, \infty}$.
\item
The energy functional $\E_\alpha$ is of class $C^1$.
\end{enumerate}
\end{lem}

\subsection{Main results}
\label{intro:main-results}

Our first result establishes sufficient conditions for the existence of \emph{ground states} of \eqref{eqn:delta-SN}, i.e., solutions to the following minimization problem:
\begin{equation}
\label{eqn:minimization-problem}
\begin{cases}
	\E_\alpha \parens{u}
	=
	m_\alpha \parens{c}
	:=
	\inf_{v \in \sphere_\alpha \parens{c}} \E_\alpha \parens{v};
	\\
	u \in \sphere_\alpha \parens{c},
\end{cases}
\end{equation}
where
$
\sphere_\alpha \parens{c}
:=
\set{u \in X_\alpha: \norm{u}_{L^2}^2 = c}
$.

\begin{thm}
\label{thm:ground-states}
Suppose that $\alpha \in \real$, $c > 0$ and one of the following conditions is satisfied:
\begin{enumerate}
\item
$\beta \leq 0$ and $p > 2$; \quad \quad
\item
$\beta > 0$ and $2 < p < 4$; \quad \quad
\item
$\beta > 0$, $p = 4$ and
$
c
<
2 / \parens{
	\beta \widetilde{K}_{\GN} \parens{p}
}
$,
\end{enumerate}
where $\widetilde{K}_{\GN} \parens{p}$ denotes the positive constant for which the Gagliardo--Nirenberg inequality in Proposition \ref{ABCT22:Proposition-2.2} holds. The following implications are satisfied.
\begin{enumerate}
\item
If $\parens{u_n}_{n \in  \nat}$ is a minimizing sequence of
$\E_\alpha|_{\sphere_\alpha \parens{c}}$, then it admits a subsequence that converges in $X_\alpha$ to a solution to \eqref{eqn:minimization-problem}.
\item
If $u$ solves \eqref{eqn:minimization-problem}, then $\Q \parens{u} \neq 0$.
\end{enumerate}
\end{thm}

The general argument used to prove the theorem follows the ideas in Cingolani \& Jeanjean's \cite{cingolaniStationaryWavesPrescribed2019}. On the other hand, we are interested in a problem that is not invariant by translation because $\op$ models the action of the formal operator
$- \Delta + \alpha^{- 1} \delta_0$. As such, we do not obtain a convergent sequence of the form
$\parens{u_n \parens{\cdot - x_n}}_{n \in \nat}$
as in \cite[Theorem 1.1]{cingolaniStationaryWavesPrescribed2019}, which is coherent with the fact that the Hilbert space $X_\alpha$ does not contain translates of functions in $X_\alpha \setminus X$.

Let us sketch the main difference with the arguments used to prove \cite[Theorem 1.1]{cingolaniStationaryWavesPrescribed2019}.
Suppose that $\parens{u_n}_{n \in \nat}$ is a minimizing sequence of
$\E_\alpha|_{\sphere_\alpha \parens{c}}$.
As in Adami, Boni, Carlone \& Tentarelli's \cite{adamiGroundStatesPlanar2022, adamiExistenceStructureRobustness2022}, the first key point is that
\eqref{eqn:minimization-problem} is energetically convenient with respect to the corresponding problem in absence of a point interaction, that is,
\begin{equation}
\label{eqn:minimization-problem-in-W1,2}
\begin{cases}
	\E \parens{\phi}
	=
	m \parens{c}
	:=
	\inf_{\psi \in \sphere \parens{c}}
	\E \parens{\psi};
	\\
	\phi \in \sphere \parens{c},
\end{cases}
\end{equation}
where
$
\sphere \parens{c}
:=
\set{\phi \in X: \norm{\phi}_{L^2}^2 = c}
$ (see Lemma \ref{lem:m_alpha(c)<m(c)}). In particular, it follows that
\begin{equation}
\label{intro:eqn:liminf-Q(u_n)>0}
\liminf_{n \to \infty}
	\abs*{\Q \parens{u_n}}
>
0.
\end{equation}

With an argument reminiscent of the proof of Lions' concentration-compactness lemma \cite[Lemma I.1]{lionsConcentrationcompactnessPrincipleCalculus1984},
we obtain a set
$\set{x_n}_{n \in \nat} \subset \real^2$
such that
$\parens{u_n \parens{\cdot - x_n}}_{n \in \nat}$
is relatively compact in $L^2$. On the other hand, the sequence
$\parens{u_n \parens{\cdot - x_n}}_{n \in \nat}$
only stays in $X_\alpha$ when either
$\set{u_n}_{n \in \nat} \subset X$
or
$
\set{x_n}_{n \in \nat} = \set{0} \subset \real^2
$.
Due to \eqref{intro:eqn:liminf-Q(u_n)>0}, we cannot expect the inclusion
$\set{u_n}_{n \in \nat} \subset X$
to hold. Bounded sequences in $\real^2$ are relatively compact, so either
$\set{x_n}_{n \in \nat}$ has a divergent subsequence or
$\parens{u_n}_{n \in \nat}$ is relatively compact in $L^2$. In the proof of Lemma \ref{lem:relatively-compact-in-L^2}, we show that the divergence $\abs{x_n} \to \infty$ as
$n \to \infty$ contradicts \eqref{intro:eqn:liminf-Q(u_n)>0}, hence the relative compactness of
$\parens{u_n}_{n \in \nat}$ in $L^2$. The rest of the proof follows Cingolani \& Jeanjean's method.

Before proceeding to the next main result, let us comment on the motivation for considering the extension of $\uop$ introduced in Section \ref{sect:extension}. Earlier investigations of nonlinear elliptic equations involving a point interaction used the fact that if
$u = \phi + q G_\lambda$ is a critical point of an energy functional
$W^{1, 2}_\alpha \to \real$, then
$u \in \Dom \parens{\uop}$ (see \cite{adamiExistenceStructureRobustness2022, adamiGroundStatesPlanar2022, depaularamosMinimizersMassconstrainedFunctionals2024, georgievStandingWavesGlobal2024}, for instance). In the context of this paper, it is not clear whether critical points of
$\E_\alpha|_{\sphere_\alpha \parens{c}}$
are in $\Dom \parens{\uop}$. Indeed, it suffices to argue as in the proof of \cite[Proposition 2.3]{cingolaniPlanarSchrodingerPoisson2016}
to deduce that if $\phi \in X$, then
\[
w_\phi \parens{x} + \norm{\phi}_{L^2}^2 \log x
\xrightarrow[\abs{x} \to \infty]{}
0.
\]
As such, we consider an extension of $\uop$ to a subspace of $L^2_{\loc}$ because the next result will show that we can at least prove that critical points of
$\E_\alpha|_{\sphere_\alpha \parens{c}}$
are automatically in $\Dom \parens{\op}$.

Finally, our last result is that critical points of
$\E_\alpha|_{\sphere_\alpha \parens{c}}$
are naturally associated with standing waves of the evolution equation
\begin{equation}
\label{eqn:evolution-delta-SN}
\iu \psi' \parens{t}
=
\op \psi \parens{t}
-
w_{\psi \parens{t}} \psi \parens{t}
-
\beta
\psi \parens{t}
\abs{\psi \parens{t}}^{p - 2}.
\end{equation}

\begin{prop}
\label{prop:regularity}
Suppose that $\alpha \in \real$, $p > 2$,
$c > 0$, $u$ is a critical point of
$\E_\alpha|_{\sphere_\alpha \parens{c}}$
and $\omega \in \real$ denotes its associated Lagrange multiplier, i.e.,
\begin{equation}
\label{lem:regularity:1}
\angles*{
	\E_\alpha' \parens{u},
	v
}_{X_\alpha^*, X_\alpha}
+
\omega
\Re \angles{u, v}_{L^2}
=
0
\quad \text{for every} \quad
v \in X_\alpha.
\end{equation}
Then
$u \in \Dom \parens{\op}$
and the function
\[
\real \ni t
\mapsto
e^{\iu \omega t} u \in \Dom \parens{\op}
\]
is a \emph{standing wave} of \eqref{eqn:evolution-delta-SN} in the sense that
\begin{equation}
\label{lem:regularity:2}
\op u + \omega u
=
w_u u
+
\beta u \abs{u}^{p - 2}
\quad \text{a.e. on} \quad
\real^2.
\end{equation}
\end{prop}

\subsection*{Acknowledgment}

\sloppy
This study was financed, in part, by the São Paulo Research Foundation (FAPESP), Brasil. Process Number \#2024/20593-0.

\section{Preliminaries}

\subsection{Properties of the Green's function $G_\lambda$}

By using the properties of the Fourier transform, an explicit computation shows that
\begin{equation}
\label{eqn:mass-of-G_lambda}
\norm{G_\lambda}_{L^2}^2
=
\frac{1}{4 \pi \lambda}.
\end{equation}
The result that follows is an adaptation of
\cite[Lemma 2.3]{depaularamosMinimizersMassconstrainedFunctionals2024}
to dimension 2.

\begin{lem}
\label{lem:Dirac-identity}
The following identity is satisfied:
\[
\int
	\parens*{- \Delta \phi \parens{x} + \lambda \phi \parens{x}}
	G_\lambda \parens{x}
\dif x
=
\phi \parens{0}
\]
for every $\lambda > 0$ and
$\phi \in W^{2, 2}_{\loc}$.
\end{lem}
\begin{proof}
It follows from the Chain Rule that
\begin{equation}
\label{lem:Dirac-identity:1}
- G_\lambda \Delta \phi
=
- \phi \Delta G_\lambda
+
\mathrm{div} \parens{\phi \nabla G_\lambda - G_\lambda \nabla \phi}.
\end{equation}
Consider an $\eps > 0$. Due to \eqref{lem:Dirac-identity:1} and the Divergence Theorem \cite[Theorem 6.3.5]{willemFunctionalAnalysisFundamentals2022},
\begin{multline*}
-
\int_{\real^2 \setminus B_\eps \parens{0}}
	G_\lambda \parens{x} \Delta \phi \parens{x}
\dif x
=
-
\int_{\real^2 \setminus B_\eps \parens{0}}
	\phi \parens{x} \Delta G_\lambda \parens{x}
\dif x
+
\\
-
\frac{\sqrt{\lambda}}{2 \pi}
K_0' \parens*{\sqrt{\lambda} \eps}
\int_{\eps \mathbb{S}^1}
	\phi \parens{x}
\dif \sigma_\eps \parens{x}
+
\frac{K_0 \parens*{\sqrt{\lambda} \eps}}
	{2 \pi}
\int_{\eps \mathbb{S}^1}
	\frac{\nabla \phi \parens{x} \cdot x}{\abs{x}}
\dif \sigma_\eps \parens{x}
\end{multline*}
where $\sigma_\eps$ denotes the surface measure of $\eps \mathbb{S}^1$. A change of variable shows that
\begin{multline*}
-
\int_{\real^2 \setminus B_\eps \parens{0}}
	G_\lambda \parens{x} \Delta \phi \parens{x}
\dif x
=
-
\int_{\real^2 \setminus B_\eps \parens{0}}
	\phi \parens{x} \Delta G_\lambda \parens{x}
\dif x
+
\\
-
\frac{\eps \sqrt{\lambda}}{2 \pi}
K_0' \parens*{\sqrt{\lambda} \eps}
\int_{\mathbb{S}^1}
	\phi \parens{\eps x}
\dif \sigma_1 \parens{x}
+
\frac{\eps}{2 \pi}
K_0 \parens*{\sqrt{\lambda} \eps}
\int_{\mathbb{S}^1}
	\nabla \phi \parens{\eps x} \cdot x
\dif \sigma_1 \parens{x}.
\end{multline*}
By summing
$
\lambda
\int_{\real^2 \setminus B_\eps \parens{0}}
	G_\lambda \parens{x} \phi \parens{x}
\dif x
$
to both sides, we obtain
\begin{multline}
\label{lem:Dirac-identity:1.5}
\int_{\real^2 \setminus B_\eps \parens{0}}
	\parens*{
		- \Delta \phi \parens{x}
		+
		\lambda \phi \parens{x}
	}
	G_\lambda \parens{x}
\dif x
=
\\
-
\frac{\eps \sqrt{\lambda}}{2 \pi}
K_0' \parens*{\sqrt{\lambda} \eps}
\int_{\mathbb{S}^1}
	\phi \parens{\eps x}
\dif \sigma_1 \parens{x}
+
\frac{\eps}{2 \pi}
K_0 \parens*{\sqrt{\lambda} \eps}
\int_{\mathbb{S}^1}
	\nabla \phi \parens{\eps x} \cdot x
\dif \sigma_1 \parens{x}
\end{multline}
because
$- \Delta G_\lambda + \lambda G_\lambda = 0$
on
$\real^2 \setminus \set{0}$.
As $\phi \in W^{2, 2}_\loc$, it follows from the Trace Inequality \cite[Theorem 6.3.3]{willemFunctionalAnalysisFundamentals2022} that
\[
\set*{
	\int_{\mathbb{S}^1}
		\nabla \phi \parens{\eps x} \cdot x
	\dif \sigma_1 \parens{x}
}_{0 < \eps \leq 1}
\]
is bounded, so
\begin{equation}
\label{lem:Dirac-identity:2}
\frac{\eps}{2 \pi}
K_0 \parens*{\sqrt{\lambda} \eps}
\int_{\mathbb{S}^1}
	\nabla \phi \parens{\eps x} \cdot x
\dif \sigma_1 \parens{x}
\xrightarrow[\eps \to 0^+]{}
0.
\end{equation}
On the other hand,
\begin{equation}
\label{lem:Dirac-identity:3}
-
\frac{\eps \sqrt{\lambda}}{2 \pi}
K_0' \parens*{\sqrt{\lambda} \eps}
\int_{\mathbb{S}^1}
	\phi \parens{\eps x}
\dif \sigma_1 \parens{x}
\xrightarrow[\eps \to 0^+]{}
\phi \parens{0}.
\end{equation}
Finally, the result follows from \eqref{lem:Dirac-identity:1.5}--\eqref{lem:Dirac-identity:3}.
\end{proof}

\subsection{A convenient decomposition of functions in $W^{1, 2}_\alpha \setminus W^{1, 2}$}
\label{sect:convenient}

We will often employ the following \emph{convenient decomposition} of functions in
$W^{1, 2}_\alpha \setminus W^{1, 2}$
inspired by the arguments in \cite{adamiExistenceStructureRobustness2022, adamiGroundStatesPlanar2022}: suppose that
$u \in W^{1, 2}_\alpha \setminus W^{1, 2}$.
The quotient
$\abs{q}^2 / \norm{u}_{L^2}^2$
is well defined because $u \not \equiv 0$,
where $q := \Q \parens{u}$. More 
precisely,
$\abs{q}^2 / \norm{u}_{L^2}^2 > 0$
because
$u \in W^{1, 2}_\alpha \setminus W^{1, 2}$
implies $q \neq 0$. As such, we can let
\[
\phi
=
u - q G_{\frac{\abs{q}^2}{\norm{u}_{L^2}^2}}
\in
W^{1, 2}.
\]
The ensuing decomposition
$
u
=
\phi + q G_{\abs{q}^2 / \norm{u}_{L^2}^2}
$
has the following useful property: due to \eqref{eqn:mass-of-G_lambda} and the triangle inequality,
\[
\parens*{1 - \frac{1}{\sqrt{4 \pi}}}
\norm{u}_{L^2}
\leq
\norm{\phi}_{L^2}
\leq
\parens*{1 + \frac{1}{\sqrt{4 \pi}}}
\norm{u}_{L^2}.
\]

\section{Existence of ground states}
\label{sect:ground-states}

In this section, we suppose that the hypotheses in Theorem \ref{thm:ground-states} are satisfied. Our first goal is to show that \eqref{eqn:minimization-problem} is energetically convenient with respect to \eqref{eqn:minimization-problem-in-W1,2} in the sense that $m_\alpha \parens{c} < m \parens{c}$.

\begin{lem}
\label{lem:m_alpha(c)<m(c)}
The inequality
$m_\alpha \parens{c} < m \parens{c}$
is satisfied. It follows that if
$\parens{u_n}_{n \in \nat}$
is a minimizing sequence of
$\E_\alpha|_{\sphere_\alpha \parens{c}}$,
then
$
\liminf_{n \to \infty}
	\abs{\Q \parens{u_n}}
>
0
$.
\end{lem}
\begin{proof}
It is clear that
$m_\alpha \parens{c} \leq m \parens{c}$
because $\E_\alpha|_X = \E$. By contradiction, suppose that
$m_\alpha \parens{c} = m \parens{c}$.
It follows from \cite{cingolaniStationaryWavesPrescribed2019}
that \eqref{eqn:minimization-problem-in-W1,2} has a solution
$\psi \in W^{2, 2}_\loc \cap X$
(the existence follows from \cite[Theorem 1.1]{cingolaniStationaryWavesPrescribed2019}, while the regularity is established in the proof of \cite[Lemma 2.7]{cingolaniStationaryWavesPrescribed2019}).
It is clear that
$
\E_\alpha \parens{\psi}
=
\E_\alpha \parens{\psi e^{\iu \theta}}
$
for $\theta \in \real$ and $\E_\alpha|_X$ is invariant by translation, so we can suppose that
$\Re \parens{\psi \parens{0}} \neq 0$.
Due to the equality
$m_\alpha \parens{c} = m \parens{c}$,
$\psi$ also solves \eqref{eqn:minimization-problem}. In particular, there exists a Lagrange multiplier $\omega \in \real$ such that
\begin{equation}
\label{proof:q-neq-0:1}
\angles*{
	\E_\alpha' \parens{\psi},
	u
}_{X_\alpha^*, X_\alpha}
+
\omega
\Re \angles{\psi, u}_{L^2}
=
0
\quad \text{for every} \quad
u \in X_\alpha.
\end{equation}
By considering the case $u = G_\lambda$ in \eqref{proof:q-neq-0:1}, we obtain
\[
\int
	\Re
	\parens*{
		-
		\lambda
		\psi \parens{x}
		+
		\omega
		\psi \parens{x}
		-
		w_\psi \parens{x} \psi \parens{x}
		-
		\beta
		\psi \parens{x}
		\abs*{\psi \parens{x}}^{p - 2}
	}
	G_\lambda \parens{x}
\dif x
=
0.
\]
On one hand, we have
\[
\int
	\Re \parens*{
		-
		\Delta \psi \parens{x}
		+
		\lambda \psi \parens{x}
	}
	G_\lambda \parens{x}
\dif x
=
0
\]
because
$
- \Delta \psi
+
\omega \psi
=
w_\psi \psi
+
\beta \psi \abs{\psi}^{p - 2}
$.
On the other hand, an application of Lemma \ref{lem:Dirac-identity} shows that
\[
\int
	\Re \parens*{
		-
		\Delta \psi \parens{x}
		+
		\lambda \psi \parens{x}
	}
	G_\lambda \parens{x}
\dif x
=
\Re \psi \parens{0}
\neq
0.
\]
We obtained a contradiction, hence the result.
\end{proof}

The next two lemmata study the link between boundedness of $\V_1$ and relative compactness in $L^2$. We begin with an adaptation of \cite[Lemma 2.5]{cingolaniStationaryWavesPrescribed2019} that is proved analogously.

\begin{lem}
\label{lem:CJ-Lemma-2.5}
Suppose that
$
\set{u_n}_{n \in \nat}
\subset
\sphere_\alpha \parens{c}
$.
If there exists $\eps \in \ooi{0, c}$ such that given $R > 0$, the inequality
\[
\liminf_{n \to \infty}
\sup_{x \in \real^2}
\int_{B_R \parens{x}}
	\abs*{u_n \parens{y}}^2
\dif y
\leq
c - \eps
\]
is satisfied, then
$
\limsup_{n \to \infty} \V_1 \parens{u_n}
=
\infty
$.
\end{lem}

Next, we develop a result in the spirit of \cite[Lemma 2.6]{cingolaniStationaryWavesPrescribed2019}.

\begin{lem}
\label{lem:relatively-compact-in-L^2}
Let
$
\set{u_n}_{n \in \nat}
\subset
\sphere_\alpha \parens{c}
$
be such that
$\set{q_n := \Q \parens{u_n}}_{n \in \nat}$
is bounded and bounded away from zero. Given
$n \in \nat$, let
\[
\phi_n
=
u_n - q_n G_{\frac{\abs{q_n}^2}{c}}
\in
X.
\]
If $\set{\V_1 \parens{u_n}}_{n \in \nat}$
is bounded and $\set{\phi_n}_{n \in \nat}$ is bounded in $W^{1, 2}$, then
$\set{u_n}_{n \in \nat}$ is relatively compact in $L^2$.
\end{lem}
\begin{proof}
Lemma \ref{lem:CJ-Lemma-2.5} is analogous to \cite[Lemma 2.5]{cingolaniStationaryWavesPrescribed2019}, so it suffices to argue precisely as in the proof of \cite[Lemma 2.6]{cingolaniStationaryWavesPrescribed2019} to deduce that there exists
$\set{x_n}_{n \in \nat} \subset \real^2$
such that
$\parens{u_n \parens{\cdot - x_n}}_{n \in \nat}$
is relatively compact in $L^2$.

We only have to show that
$\set{x_n}_{n \in \nat}$
is bounded to conclude that
$\parens{u_n}_{n \in \nat}$
is relatively compact in $L^2$. By contradiction, suppose that, up to subsequence,
$\abs{x_n} \to \infty$ as $n \to \infty$.

Let us show that
\begin{equation}
\label{eqn:q_n-to-zero}
q_n
G_{\frac{\abs{q_n}^2}{c}} \parens{\cdot - x_n}
\xrightharpoonup[n \to \infty]{L^2}
0.
\end{equation}
Indeed,
$
\set{
	q_n G_{\abs{q_n}^2 / c} \parens{\cdot - x_n}
}_{n \in \nat}
$
is bounded in $L^2$ due to \eqref{eqn:mass-of-G_lambda} and it is easy to check that
\[
\angles*{
	q_n
	G_{\frac{\abs{q_n}^2}{c}}
		\parens{\cdot - x_n}
	,
	\psi
}_{L^2}
\xrightarrow[n \to \infty]{}
0
\]
for every $\psi \in C_c^\infty$. As such, the result follows from \cite[Proposition 5.1.2]{willemFunctionalAnalysisFundamentals2022}.

The set
$\set{\phi_n \parens{\cdot - x_n}}_{n \in \nat}$
is bounded in $W^{1, 2}$, so there exists
$\phi_\infty \in W^{1, 2}$
such that, up to subsequence,
\begin{equation}
\label{eqn:phi_n-to-phi_infty}
\phi_n \parens{\cdot - x_n}
\xrightharpoonup[n \to \infty]{W^{1, 2}}
\phi_\infty.
\end{equation}
It follows from \eqref{eqn:q_n-to-zero} and \eqref{eqn:phi_n-to-phi_infty} that
\[
u_n \parens{\cdot - x_n}
\xrightharpoonup[n \to \infty]{L^2}
\phi_\infty.
\]
The set
$\set{u_n \parens{\cdot - x_n}}_{n \in \nat}$
is relatively compact in $L^2$, so the convergence
\[
u_n \parens{\cdot - x_n}
=
\phi_n \parens{\cdot - x_n}
+
q_n
G_{\frac{\abs{q_n}^2}{c}} \parens{\cdot - x_n}
\xrightarrow[n \to \infty]{L^2}
\phi_\infty
\]
holds up to subsequence. Equivalently,
\[
\phi_n - \phi_\infty \parens{\cdot + x_n}
+
q_n
G_{\frac{\abs{q_n}^2}{c}}
\xrightarrow[n \to \infty]{L^2}
0.
\]
The set $\set{q_n}_{n \in \nat}$ is bounded and bounded away from zero, so there exists
$q_\infty \in \complex \setminus \set{0}$
such that, up to subsequence, $q_n \to q_\infty$ as
$n \to \infty$. Therefore,
\begin{equation}
\label{eqn:phi_n-phi_infty-in-L2}
\phi_n - \phi_\infty \parens{\cdot + x_n}
\xrightarrow[n \to \infty]{L^2}
- q_\infty G_{\frac{\abs{q_\infty}^2}{c}}.
\end{equation}
By comparing \eqref{eqn:phi_n-to-phi_infty} and \eqref{eqn:phi_n-phi_infty-in-L2}, we deduce that $q_\infty = 0$. We just obtained a contradiction, hence the result.
\end{proof}

In view of the previous results, we can prove the following minor variations of \cite[Lemmata 3.1 and 3.2]{cingolaniStationaryWavesPrescribed2019} with analogous arguments.

\begin{lem}
\label{lem:CJ-Lemma-3.1}
Suppose that
$
\set{u_n}_{n \in \nat}
\subset
\sphere_\alpha \parens{c}
$,
$\set{u_n}_{n \in \nat}$
is bounded in $W^{1, 2}_\alpha$,
$\set{\E_\alpha \parens{u_n}}_{n \in \nat}$ is bounded from above,
$\set{\Q \parens{u_n}}_{n \in \nat}$
is bounded and bounded away from zero. Then, up to subsequence,
$\parens{u_n}_{n \in \nat}$
is weakly convergent on $X_\alpha$.
\end{lem}

\begin{lem}
\label{lem:CJ-Lemma-3.2}
Suppose that
$
\set{u_n}_{n \in \nat}
\subset
\sphere_\alpha \parens{c}
$ 
is such that $u_n \rightharpoonup u$ in $X_\alpha$ as $n \to \infty$. Then
$
\E_\alpha \parens{u}
\leq
\liminf_{n \to \infty}
\E_\alpha \parens{u_n}
$.
If we suppose further that
$\E_\alpha \parens{u_n} \to \E_\alpha \parens{u}$
as $n \to \infty$, then $u_n \to u$ in $X_\alpha$ as $n \to \infty$.
\end{lem}

At this point, we need to recall two useful inequalities. The following lemma consists of a particular case of the Hardy--Littlewood--Sobolev inequality \cite[4.3 Theorem]{liebAnalysis2001}.

\begin{prop}
\label{prop:HLS}
Suppose that $r, s > 1$ are such that
$r^{- 1} + s^{- 1} = 3 / 2$. Then there exists a positive constant
$K_{\HLS} \parens{r}$
such that
\[
\abs*{
\int \int
	\frac{f \parens{x} g \parens{y}}{\abs{x - y}}
\dif x \dif y
}
\leq
K_{\HLS} \parens{r}
\norm{f}_{L^r}
\norm{g}_{L^s}
\]
for every $f \in L^r$ and $g \in L^s$.
\end{prop}

The next result contains a recent generalization of the classical Gagliardo--Nirenberg inequality.

\begin{prop}[{\cite[Proposition 2.2]{adamiGroundStatesPlanar2022}}]
\label{ABCT22:Proposition-2.2}
Suppose that $p \in \ooi{2, \infty}$.
\begin{enumerate}
\item
There exists a positive constant
$K_{\GN} \parens{p}$
such that
\[
\norm{u}_{L^p}^p
\leq
K_{\GN} \parens{p}
\parens*{
	\norm{\phi}_{\dot{W}^{1, 2}}^{p - 2}
	\norm{\phi}_{L^2}^2
	+
	\frac{\abs{q}^p}{\lambda}
}
\]
for every
$u = \phi + q G_\lambda \in W^{1, 2}_\alpha$.
\item
There exists a positive constant
$
\widetilde{K}_{\GN} \parens{p}
$
such that
\[
\norm{u}_{L^p}^p
\leq
\widetilde{K}_{\GN} \parens{p}
\parens*{
	\norm{\phi}_{\dot{W}^{1, 2}}^{p - 2}
	+
	\abs{q}^{p - 2}
}
\norm{u}_{L^2}^2
\]
for every
$
u
=
\phi + q G_{\abs{q}^2 / \norm{u}_{L^2}^2}
\in
W^{1, 2}_\alpha \setminus W^{1, 2}
$.
\end{enumerate}
\end{prop}

Now, we collect lower bounds for
$\E_\alpha$ obtained by means of Propositions \ref{prop:HLS} and \ref{ABCT22:Proposition-2.2}.

\begin{lem}
\label{lem:lower-bound-for-E_alpha}
\begin{enumerate}
\item
If $\beta \leq 0$, then
\[
\E_\alpha \parens{\phi}
\geq
\frac{1}{2}
\norm{\phi}_{\dot{W}^{1, 2}}^2
-
\frac{1}{4}
K_{\HLS} \parens*{\frac{4}{3}}
K_{\GN} \parens*{\frac{8}{3}}
\norm{\phi}_{L^2}^3
\norm{\phi}_{\dot{W}^{1, 2}}
\]
for every $\phi \in X$ and
\begin{align*}
\E_\alpha \parens{u}
&\geq
\frac{1}{2}
\norm{\phi}_{\dot{W}^{1, 2}}^2
+
\frac{
	\norm{\phi}_{L^2}^2
	\abs{q}^2
}{2 \norm{u}_{L^2}^2}
+
\\
&+
\frac{1}{2}
\parens*{
	\alpha
	+
	\frac{\gamma}{2 \pi}
	+
	\frac{1}{2 \pi}
	\log \parens*{
		\frac{\abs{q}}{2 \norm{u}_{L^2}}
	}
	-
	1
}
\abs{q}^2
+
\\
&-
\frac{1}{2 \sqrt{2}}
K_{\HLS} \parens*{\frac{4}{3}}
\widetilde{K}_{\GN} \parens*{\frac{8}{3}}
\norm{u}_{L^2}^3
\parens*{
	\norm{\phi}_{\dot{W}^{1, 2}}
	+
	\abs{q}
}
\end{align*}
for every
$
u
=
\phi + q G_{\abs{q}^2 / \norm{u}_{L^2}^2}
\in
X_\alpha \setminus X
$.
\item
If $\beta > 0$, then
\begin{align*}
\E_\alpha \parens{\phi}
&\geq
\frac{1}{2}
\norm{\phi}_{\dot{W}^{1, 2}}^2
-
\frac{1}{4}
K_{\HLS}\parens*{\frac{4}{3}}
K_{\GN} \parens*{\frac{8}{3}}
\norm{\phi}_{L^2}^3
\norm{\phi}_{\dot{W}^{1, 2}}
+
\\
&-
\frac{\beta}{p}
K_{\GN} \parens{p}
\norm{\phi}_{L^2}^2
\norm{\phi}_{\dot{W}^{1, 2}}^{p - 2}
\end{align*}
for every $\phi \in X$ and
\begin{align*}
\E_\alpha \parens{u}
&\geq
\frac{1}{2}
\norm{\phi}_{\dot{W}^{1, 2}}^2
+
\frac{
	\norm{\phi}_{L^2}^2
	\abs{q}^2
}{2 \norm{u}_{L^2}^2}
+
\\
&+
\frac{1}{2}
\parens*{
	\alpha
	+
	\frac{\gamma}{2 \pi}
	+
	\frac{1}{2 \pi}
	\log \parens*{
		\frac{\abs{q}}{2 \norm{u}_{L^2}}
	}
	-
	1
}
\abs{q}^2
+
\\
&-
\frac{1}{2 \sqrt{2}}
K_{\HLS} \parens*{\frac{4}{3}}
\widetilde{K}_{\GN} \parens*{\frac{8}{3}}
\norm{u}_{L^2}^3
\parens*{
	\norm{\phi}_{\dot{W}^{1, 2}}
	+
	\abs{q}
}
+
\\
&-
\frac{\beta}{p}
\widetilde{K}_{\GN} \parens{p}
\norm{u}_{L^2}^2
\parens*{
	\norm{\phi}_{\dot{W}^{1, 2}}^{p - 2}
	+
	\abs{q}^{p - 2}
}
\end{align*}
for every
$
u
=
\phi + q G_{\abs{q}^2 / \norm{u}_{L^2}^2}
\in
X_\alpha \setminus X
$.
\end{enumerate}
\end{lem}

We proceed to the proof of the theorem.

\begin{proof}[Proof of Theorem \ref{thm:ground-states}]
Let
$\parens{u_n}_{n \in \nat}$
denote a minimizing sequence of
$\E_\alpha|_{\sphere_\alpha \parens{c}}$.
In particular,
\begin{equation}
\label{eqn:E_alpha(u_n)-bounded-from-above}
\set*{\E_\alpha \parens{u_n}}_{n \in \nat}
~
\text{is bounded from above}
\end{equation}
and it follows from Lemma \ref{lem:m_alpha(c)<m(c)} that
\begin{equation}
\label{eqn:liminf-Q(u_n)>0}
\liminf_{n \to \infty}
	\abs*{q_n}
>
0,
\end{equation}
where $q_n := \Q \parens{u_n}$. In particular, we can suppose that $q_n \neq 0$ for every
$n \in \nat$ (which holds up to discarding a finite number of indices). Let
\[
\phi_n
=
u_n
-
q_n G_{\frac{\abs{q_n}^2}{c}}
\in
X.
\]
Due to \eqref{eqn:E_alpha(u_n)-bounded-from-above} and Lemma \ref{lem:lower-bound-for-E_alpha}, we deduce that
\begin{equation}
\label{eqn:phi_n-bounded-in-dotW}
\set{\phi_n}_{n \in \nat}
~ \text{is bounded in} ~
\dot{W}^{1, 2}
\end{equation}
and
\begin{equation}
\label{eqn:q_n-bounded}
\set{q_n}_{n \in \nat}
~
\text{is bounded}.
\end{equation}
In view of \eqref{eqn:phi_n-bounded-in-dotW} and the considered decomposition of $u_n$,
\begin{equation}
\label{eqn:phi_n-bounded-in-W}
\set{\phi_n}_{n \in \nat}
~ \text{is bounded in} ~
W^{1, 2}
\end{equation}
(see Section \ref{sect:convenient}). Due to \eqref{eqn:q_n-bounded} and \eqref{eqn:phi_n-bounded-in-W}, we obtain
\begin{equation}
\label{eqn:u_n-bounded-in-W}
\set{u_n}_{n \in \nat}
~ \text{is bounded in} ~
W^{1, 2}_\alpha.
\end{equation}
In view of \eqref{eqn:E_alpha(u_n)-bounded-from-above}, \eqref{eqn:liminf-Q(u_n)>0}, \eqref{eqn:q_n-bounded} and \eqref{eqn:u_n-bounded-in-W}, it follows from Lemma \ref{lem:CJ-Lemma-3.1} that there exists
$u_\infty \in X_\alpha$ such that, up to subsequence,
\begin{equation}
\label{eqn:u_n-rightharpoonup-u_infty}
u_n
\xrightharpoonup[n \to \infty]{X_\alpha}
u_\infty.
\end{equation}
The natural embedding
$X_\alpha \hookrightarrow L^2$
is compact (see Lemma \ref{lem:CW-2.2}), so
\begin{equation}
\label{eqn:u_infty-in-S_alpha(c)}
u_\infty \in \sphere_\alpha \parens{c}.
\end{equation}
Due to \eqref{eqn:u_n-rightharpoonup-u_infty}, Lemma \ref{lem:CJ-Lemma-3.2} shows that
\[
\E_\alpha \parens{u_\infty}
\leq
\liminf_{n \to \infty} \E_\alpha \parens{u_n}
=
m_\alpha \parens{c}.
\]
In view of \eqref{eqn:u_infty-in-S_alpha(c)} and the definition of $m_\alpha \parens{c}$, we deduce that
\begin{equation}
\label{eqn:E_alpha(u_infty)=m_alpha(c)}
\E_\alpha \parens{u_\infty} = m_\alpha \parens{c},
\end{equation}
that is, $u_\infty$ is a solution to \eqref{eqn:minimization-problem}. To finish, the convergence $u_n \to u_\infty$ in $X_\alpha$ as
$n \to \infty$ also follows from Lemma \ref{lem:CJ-Lemma-3.2} and the inequality
$\Q \parens{u_\infty} \neq 0$
is a corollary of Lemma \ref{lem:m_alpha(c)<m(c)}.
\end{proof}

\section{Proof of Proposition \ref{prop:regularity}}
\label{sect:regularity}

\paragraph{Proof that
$u \in \Dom \parens{\op}$.}
Fix $\lambda > 0$ and let
$\phi = u - q G_\lambda$,
where
$q := \Q \parens{u}$.
We want to prove that
$\phi \in W^{2, 2}_{\loc}$
and
$
\phi \parens{0}
=
\parens{\alpha + \theta_\lambda} q
$.

\subparagraph{Proof that
$\phi \in W^{2, 2}_{\loc}$.}
It suffices to argue as in the proof of
\cite[Proposition 2.3]{cingolaniPlanarSchrodingerPoisson2016} to deduce that
$w_u \in L^\infty_{\loc}$.
As such, we can use the embeddings
$
X_\alpha \hookrightarrow 
L^2, L^{2 \parens{p - 1}}
$
to deduce that
\[
\beta u \abs{u}^{p - 2}
-
\omega u
-
w_u u
+
q \lambda G_\lambda
\in
L^2_{\loc}.
\]
By hypothesis,
$\phi \in X \subset W^{1, 2}$
is a weak solution to
\[
- \Delta \phi
=
\beta u \abs{u}^{p - 2}
-
\omega u
-
w_u u
+
q \lambda G_\lambda,
\]
so the result follows from the interior regularity result \cite[Theorem 5.1]{chenSecondOrderElliptic1998}.

\subparagraph{Proof that
$
\phi \parens{0}
=
\parens{\alpha + \theta_\lambda} q
$.}
It follows from \eqref{lem:regularity:1} that
\begin{align*}
0
&=
\angles*{
	\E_\alpha' \parens{u},
	G_\lambda
}_{X_\alpha^*, X_\alpha}
+
\omega
\Re
\int
	u \parens{x}
	G_\lambda \parens{x}
\dif x;
\\
&=
\int
	\parens*{\parens*{
			\omega
			-
			w_u \parens{x}
			-
			\beta \abs*{u \parens{x}}^{p - 2}
		}
		\Re \parens*{u \parens{x}}
		-
		\lambda
		\Re \parens{q}
		G_\lambda \parens{x}
	}
	G_\lambda \parens{x}
\dif x
+
\\
&-
\lambda
\int
	\Re \parens*{\phi \parens{x}}
	G_\lambda \parens{x}
\dif x
+
\parens{\alpha + \theta_\lambda}
\Re \parens*{\overline{q}}.
\end{align*}
That is,
\begin{multline*}
\Re
\int
	\parens*{
		-
		\omega
		u \parens{x}
		+
		w_u \parens{x}
		u \parens{x}
		+
		\beta
		\abs*{u \parens{x}}^{p - 2}
		u \parens{x}
		+
		\lambda
		\Re \parens{q}
		G_\lambda \parens{x}
	}
	G_\lambda \parens{x}
\dif x
+
\\
+
\Re
\int
	\lambda
	\phi \parens{x}
	G_\lambda \parens{x}
\dif x
=
\parens{\alpha + \theta_\lambda}
\Re \parens*{\overline{q}}.
\end{multline*}
As $\phi \in W^{2, 2}_{\loc}$, we obtain
\begin{equation}
\label{lem:regularity:3}
-
\Delta \phi
=
-
\omega u
+
w_u u
+
\beta u \abs{u}^{p - 2}
+
q \lambda G_\lambda
\quad \text{a.e. on} \quad
\real^2.
\end{equation}
Therefore,
\[
\int
	\Re \parens*{
		-
		\Delta \phi \parens{x}
		+
		\lambda \phi \parens{x}
	}
	G_\lambda \parens{x}
\dif x
=
\parens{\alpha + \theta_\lambda}
\Re \parens{q}.
\]
It suffices to repeat the argument with
$\iu G_\lambda$ in place of $G_\lambda$ to deduce that
\[
\int
	\parens*{
		-
		\Delta \phi \parens{x}
		+
		\lambda \phi \parens{x}
	}
	G_\lambda \parens{x}
\dif x
=
\parens{\alpha + \theta_\lambda} q.
\]
Finally, the result follows from Lemma \ref{lem:Dirac-identity}.

\paragraph{Proof of \eqref{lem:regularity:2}.}
Corollary of \eqref{lem:regularity:3}.

\qed

\sloppy
\printbibliography
\end{document}